\documentclass[11pt]{amsart}

\usepackage{pb-diagram}

     \usepackage{xcolor}
\usepackage[top=1.5in, bottom=1in, left=0.5in, right=3in]{geometry}
\usepackage{amsfonts}
\usepackage{amsmath}
\usepackage{amssymb}
\usepackage{verbatim}
\usepackage{color}

\newcommand{\al}{\alpha}
\newcommand{\be}{\beta}

\newtheorem{theorem}{Theorem}[section]

\newtheorem{lemma}{Lemma}[section]

\newtheorem{remark}{Remark}[section]

\begin{document}

\title{A remark on K\"ahler metrics with conical singularities along a simple normal crossing divisor}

\author{Ved V. Datar$^*$, Jian Song$^\dagger$ }

\address{$*$ Department of Mathematics, Rutgers University, Piscataway, NJ 08854}

\email{veddatar@math.rutgers.edu}

\address{$\dagger$ Department of Mathematics, Rutgers University, Piscataway, NJ 08854}

\email{jiansong@math.rutgers.edu}

\thanks{Research supported in
part by National Science Foundation grants DMS-0847524 and the graduate fellowship of Rutgers University.}

\begin{abstract} 
Recently it was shown by H. Guenancia and  M. P$\breve{\text{a}}$un that a singular metric satisfying the conical K\"ahler-Einstein equation with a simple normal crossing divisor is equivalent to a conical metric along that divisor. In this note, we present an alternative proof of their theorem.
\end{abstract}

\maketitle

%

\maketitle

Let $(X,\omega)$ be an $n$-dimensional K\"ahler manifold with a smooth K\"ahler metric $\omega$. We fix a divisor $D = \sum_{j=1}^{N}(1-\be_j)D_j$, where $\be_j \in (0,1)$ and $D_j$'s are irreducible smooth divisors. We further assume that $D$ is a simple normal crossing divisor i.e   for any $p \in Supp(D)$ lying in the intersection of exactly $k$ divisors $D_1,\cdots, D_k$, there exists a coordinate chart $(U_p,\{z_j\})$ containing $p$, such that $D_j|_{U_p} = \{z_j =0\}$ for $j=1,\cdots,k$.

 If $s_j$ is the defining section of $D_j$ and $h_j$ is any smooth metric on the line bundle induced by $D_j$, then for sufficiently small $\epsilon_j>0$, $\theta_j = \omega+\epsilon_j \sqrt{-1} \partial\bar\partial|s_j|_{h_j}^{2\be_j}$ gives a K\"ahler metric on $X\backslash Supp (D_j)$ with cone angle $2\pi\be_j$ along $D_j$. Now, set 
\begin{equation}
\theta =\sum_{j=1}^{N}{\theta_j}.
\end{equation}
\noindent Then, $\theta$ is a smooth K\"ahler metric on $X \backslash Supp(D)$. Moreover, for any $p \in Supp(D)$ and any coordinate chart $(U_p,\{z_j\})$ as above, $\theta$ is uniformly equivalent to the standard cone metric
\begin{equation}
\omega_{p} = \sum_{j=1}^{k}\frac{\sqrt{-1}dz_j \wedge d\bar z_{j}}{|z_j|^{2(1-\be_j)}} + \sum_{j=k+1}^{N}{\sqrt{-1}dz_j \wedge d\bar z_{j}}.
\end{equation}

\noindent Next, for any $\lambda \in \mathbb R$, we consider the Monge-Amp\`ere equation
\begin{equation}\label{m.a}
\begin{cases}
(\omega+\sqrt{-1}\partial\bar\partial \varphi)^n = \frac{e^{- \lambda\varphi} \Omega}{\prod_{j=1}^{N}{|s_j|_{h_j}^{2(1-\be_j)}}}\\
\omega_{\varphi} = \omega + \sqrt{-1}\partial\bar\partial \varphi > 0
\end{cases}
\end{equation}
\noindent for some smooth volume form $\Omega$. By rescaling one can always assume that $\lambda = 1,0,-1$. In the case that $\lambda = 0$, we impose an additional normalization that $\sup_{M}\varphi = 0$. The above equation arises when one considers the problem of prescribing the Ricci curvature of a conical metric and was first studied by Yau in \cite{Y2}. A natural question is whether conversely, any metric $\omega_{\varphi}$ solving \eqref{m.a} is in fact conical. The answer is provided by \cite{D,B,JMR} in the case of a smooth divisor and by \cite{CGP,GP} in the general case. In the present note, we demonstrate that the situation with a simple normal crossing divisor is no harder than the one with a smooth divisor, thereby providing an alternate proof for the following theorem of H. Guenancia and  M. P$\breve{\text{a}}$un.
\begin{theorem}\cite{GP}\label{mt}
If $\varphi$ is any bounded solution to \eqref{m.a}, then there exists a constant $C>0$ such that 
\begin{equation}
C^{-1}\theta \leq \omega_{\varphi} \leq C\theta
\end{equation}
on $X \backslash Supp(D)$,  i.e., $\omega_{\varphi}$ is equivalent to a conical K\"ahler metric along $D$.
\end{theorem}
\noindent The theorem is essentially equivalent to certain second order estimates. We begin with 
\begin{lemma}
There exists a constant $a>0$ such that, for any $j \in \{1,\cdots,N\}$
\begin{equation}\label{lemma}
\omega_{\varphi} \geq a\theta_j
\end{equation}
on $X\backslash Supp(D)$.
\end{lemma}
\begin{proof}
\noindent We first assume that $\lambda = 0$. The proof in the other two cases is similar (cf. Remark \ref{remark}) with only a minor modification in the case of $\lambda = 1$. We set $f_j = \log{(\prod_{i=1, ... , N, i\neq j}{|s_i|_{h_i}^{2(1-\be_i)}})}$. Then for some constant $A>>1$, $ \sqrt{-1}\partial\bar\partial f_j > -A\omega$ as currents. By Demailly's regularization theorem \cite{De}, there exist functions $F_{j,k} \in  C^{\infty}(X)$ such that $F_{j,k} \searrow f_j$ and $\sqrt{-1}\partial\bar\partial F_{j,k} >  -A\omega$.  Now, consider the following family of Monge-Amp\`ere equations
\begin{equation}
\begin{cases}
(\omega + \sqrt{-1}\partial\bar\partial\varphi_{j,k})^n = \frac{e^{(-F_{j,k} + c_{j,k})}  \Omega}{|s_j|^{2(1-\be_j)}_{h_j}}\\
\omega_{j,k} = \omega + \sqrt{-1}\partial\bar\partial\varphi_{j,k} > 0\\
\sup_{M}\varphi_{j,k} = 0.
\end{cases}
\label{approx m.a}
\end{equation}
It is well known \cite{D, B,JMR,CDS2} that  there always exists a solution $\varphi_{j,k}$ in $C^{2,\al,\be_j}(X)$ for some $\al \in (0,1)$. Here $C^{2,\al,\be}(X)$ are the H\"older spaces defined in \cite{D}. Note that by integrating both sides of the equation, it is easy to see that the constants $c_{j,k}$ are uniformly bounded and converge to zero as $k \rightarrow \infty$. Since $\be_j \in (0,1)$, from \eqref{approx m.a} it is clear that $\omega_{j,k}^n/\Omega \in L^{1+\epsilon}(X,\Omega)$ for some $\epsilon > 0$ with uniform control over the $L^{1+\epsilon}$ norm. So, by Kolodziej's theorem \cite{K1}, the solutions $\varphi_{j,k}$ are uniformly bounded in the $C^0$ norm. In fact, since from the equation $\omega_{j,k}^n/\Omega \rightarrow \omega_{\varphi}^n/\Omega$ in $L^{1}(X,\Omega)$, by the stability of solutions of complex Monge-Amp\`ere equations \cite{K2}, $|\varphi_{j,k} - \varphi|_{C^0(X)} \rightarrow 0$, as $k\rightarrow \infty$.

To obtain second order estimates, we note that $tr_{\omega_{j,k}}\theta_j$ is bounded since $\varphi_{j,k} \in C^{2,\al,\be_j}(X)$, and so for any $\delta > 0$ and $B>0$, the quantity
\begin{equation}\label{schwarz}
Q = \log{(|s_j|_{h_j}^{2\delta}tr_{\omega_{j,k}}\theta_j}) - B\left(\varphi_{j,k} - \epsilon_j |s_j|^{2\beta_j}_{h_j} \right)
\end{equation} 
attains its maximum value at some $p_{max} \in X\backslash Supp(D_j)$. Without loss of generality, we can assume that $|s_j|_{h_j}\leq 1$ on $X$. First, it follows from \eqref{approx m.a} and the fact that $\omega \leq c \theta$ for some $c>0$, that there exists a uniform $C>0$ such that $Ric(\omega_{j,k}) > -C\theta_j$. Next, the bisectional curvature of $\theta_j$ is bounded above \cite{JMR}. Hence by the Chern-Lu inequality \cite{C,L,Y1}, there exist constants $B,C>0$ independent of $j,k$ and $\delta$ such that 
\begin{equation*}
\Delta_{\omega_{j,k}}Q \geq tr_{\omega_{j,k}}\theta_j - C.
\end{equation*}
\noindent By the maximum principle and the uniform $C^0$ estimates, there exists an $a>0$ such that $tr_{\omega_{j,k}}\theta_j  \leq a^{-1}/|s_j|_{h_j}^{2\delta}$ on $X$. Letting $\delta \rightarrow 0$,
\begin{equation*}
\omega_{j,k} \geq a\theta_j.
\end{equation*}
 \noindent Taking limit as $k\rightarrow \infty$ we prove that $ \omega_{\varphi} \geq a\theta_j$ as currents. It follows, for instance from the regularization properties of Monge-Amp\`ere flows \cite{ST}, that $\omega_{\varphi}$ is in fact smooth away from $Supp(D)$ and hence the inequality must be point-wise. This completes the proof of the lemma. 
\end{proof}
\noindent \textit{\bf{Proof of Theorem \ref{mt}}.}   By adding the lower bounds from (\ref{lemma}) for $j=1,\cdots,N$,  we prove that there exists $C>0$ such that
\begin{equation*}
\omega_{\varphi} \geq  C^{-1} \theta.
\end{equation*}
\noindent Since $\theta$ is locally equivalent to a cone metric with angle $2\pi\be_j$ along $D_j$, it is easy to check that
\begin{equation*}
\theta^n = \frac{\Omega^\prime}{\prod_{j=1}^{N}{|s_j|_{h_j}^{2(1-\be_j)}}}
\end{equation*} 
\noindent for some continuous nowhere vanishing volume form $\Omega^{\prime}$, i.e., $\theta^n$ and $\omega_{\varphi}^n$ are uniformly equivalent on $X\backslash Supp(D)$. Together with the lower bound on the metric, it directly gives the required upper bound on the metric. \qed 
\begin{remark}\label{remark}
The proof with $\lambda = -1$ can be carried out exactly as above. When $\lambda = 1$, we approximate the solution $\varphi$ on the right hand  side of \eqref{m.a} by quasi-plurisubharmonic functions as in \cite{GP} and again proceed as above.
\end{remark}

\end{document}